\documentclass[11pt]{article}

\usepackage{setspace}
\usepackage{a4wide}

\usepackage{amsthm, amssymb, amstext}
\usepackage[fleqn]{amsmath}
\usepackage{latexsym}
\usepackage[dvips]{graphicx}
\usepackage{comment}
\usepackage{hyperref}
\usepackage{mathtools}
\usepackage{enumerate} 

\usepackage[T1]{fontenc}

\newtheorem{theorem}{Theorem}
\newtheorem{lemma}{Lemma}

\theoremstyle{definition}

\usepackage{amsthm}

\title{Dirac's theorem on chordal graphs implies Brooks' theorem}
\author{Carl Feghali\thanks{Univ Lyon, EnsL, CNRS, LIP, F-69342, Lyon Cedex 07, France, email: \texttt{carl.feghali@ens-lyon.fr}}}

\date{}

\begin{document}
\maketitle

\begin{abstract}
We give yet another proof of the list-color version of Brooks' theorem that is due, independently, to Vizing and to Erd\H{o}s, Rubin and Taylor, via a famous theorem of Dirac on chordal graphs. 
 \end{abstract}

Let $G$ be a graph, and let $k$ be a non-negative integer. 
A (proper) \emph{$k$-coloring} of $G$ is a function $\varphi: V(G) \rightarrow \{1, \dots, k\}$ such that $\varphi(u) \neq \varphi(v)$ if $uv \in E(G)$. A \emph{list assignment} of a graph is a function $L$ that assigns to
each vertex $v $ a list $L(v)$ of colors. The graph $G$ is \emph{$L$-colorable} if it has a
proper coloring $f$ such that $f(v) \in L(v)$ for each vertex $ v $ of  $G$. 

A classical theorem of Brooks  in graph coloring theory draws a connection between the maximum degree of a graph and its chromatic number. 

\begin{theorem}[\cite{brooks}]\label{conj}
Let $\Delta \geq 3$ be an integer. Then every connected graph $G \not= K_{\Delta + 1}$ with maximum degree $\Delta$
 has a $\Delta$-coloring.\end{theorem}
 
Many new proofs (see, for example, \cite{cranston, rabern, zajkac}) extensions and generalizations (see, for example, \cite{cranston, krivelevich, matamala}) of Theorem \ref{conj} have appeared.

 In this note, we give another proof that uses a theorem of Dirac on chordal graphs. A graph is \emph{chordal} if every chordless cycle of the graph is a triangle. A \emph{perfect elimination ordering} of a graph $G$ is an ordering $v_1, \dots, v_n$ of $V(G)$ such that, for each $i \in [n]$, the graph induced by $\{v_j: j < i, v_iv_j \in E(G)\}$ is a clique. Dirac \cite{dirac} famously showed that a graph is chordal if, and only if, it admits a perfect elimination ordering. This immediately implies

\begin{lemma}\label{lem}
Let $G$ be a chordal graph with clique number $\omega$. Let $L$ be a list assignment of $G$ such that $|L(v)| \geq \omega$ for each $v \in V(G)$. Then $G$ has an $L$-coloring. 
\end{lemma}

\begin{proof}
Let  $v_1, \dots, v_n$ be a perfect elimination ordering of $G$; color the vertices of $G$ greedily in this order (starting with $v_1$ and moving towards $v_n$). This is possible since the number of already colored neighbors of a still uncolored vertex $v$ is always less than $|L(v)|$. 
\end{proof}

We are now ready to prove Theorem \ref{conj}; in fact, we prove its following more general list-color version due, independently, to Vizing \cite{vizing} and to Erd\H{o}s, Rubin and Taylor \cite{erdos}.

\begin{theorem}
Let $\Delta \geq 3$ and let $G \not= K_{\Delta+1}$ be a connected graph with maximum degree $\Delta$. If $L$ is a list assignment of $G$ such that $|L(v)| \geq \Delta$ for each $v \in V(G)$, then $G$ has an $L$-coloring. 
\end{theorem}

 \begin{proof}
Let $G$ be a counterexample with $|V(G)|$ as small as possible. Then $G$ is not chordal by Lemma \ref{lem}. Consider a chordless cycle $C = x_1x_2 \dots x_kx_1$ with $k \geq 4$ in $G$. Let $F = G - (C -  \{x_1, x_2, x_3\}) + (x_1, x_3)$ and $H = G - (C - \{x_2, x_3, x_4\}) + (x_2, x_{4})$.  If both $F$ and $H$ have a $K_{\Delta + 1}$, then as any $K_{\Delta+1}$ must use an added edge, all neighbors of $x_2$ outside $C$ are pairwise adjacent, and adjacent to $x_1, x_2, x_3, x_4$, which is impossible since otherwise their degrees would exceed $\Delta$. Thus, $F$ or $H$, say $F$, has no $K_{\Delta +1}$. 

Let $L'$ be the restriction of $L$ to $V(F)$. By minimality, $F$ has an $L'$-coloring $f$. Let $L^*$ be the list assignment for $C$ defined by $L^*(x_i) = L(x_i) \setminus \{f(u): (u, x_i) \in E(G) - E(C)\}$ for $i \in [k]$. Then $|L^*(x_i)| \geq 2$ for $i \in [k]$, and if $|L^*(x_1)| = |L^*(x_2)| = L^*(x_3)| = 2$, then $L^*(x_1)$, $L^*(x_2)$, $L^*(x_3)$ are not pairwise equal, otherwise $f$ would not exist as $x_1, x_2, x_3$ form a triangle in $F$. We can thus assume, without loss of generality, that $L^*(x_1)$ has a color $c_1$ such that $|L^*(x_2) \setminus c_1| \geq 2$. Then the restriction of $f$ to $G - C$ can be extended to an $L$-coloring of $G$ by giving $x_1$ the color $c_1$, $x_k$ a color $c_k \in L^*(x_k) \setminus \{c_1\}$, for $i = k-1, \dots, 3$ in order, $x_{i}$ a color $c_i \in L^*(x_i) \setminus \{c_{i+1}\}$ and $x_2$ a color $c_2 \in L^*(x_2) \setminus \{c_1, c_3\}$, which is a contradiction. \end{proof}

\noindent
\textbf{Acknowledgements.} I am grateful to Daniel W. Cranston for feedback on a previous version of this manuscript and to Dibyayan Chakraborty and Benjamin Moore for a helpful discussion.  This work was supported by the French National Research Agency under research grant ANR DIGRAPHS ANR-19-CE48-0013-01

 \bibliography{bibliography}{}
\bibliographystyle{abbrv}
 
\end{document}